\documentclass[preprint]{elsarticle}
\usepackage{amsmath, amssymb, amsthm}

\newcommand{\scrR}{\mathcal{R}}
\newcommand{\scrP}{\mathcal{P}}
\newcommand{\scrL}{\mathcal{L}}

\newcommand{\bbC}{\mathbb{C}}

\newcommand{\bbF}{\mathbb{F}}

\newcommand{\rank}{\text{rank}}
\newcommand{\tr}{\text{tr}}
\DeclareMathOperator\h{\mathsf{H}}

\DeclareMathOperator\q{\mathsf{Q}}

\DeclareMathOperator\sfO{\mathsf{O}}
\DeclareMathOperator\sfF{\mathsf{F}}
\DeclareMathOperator\sfI{\mathsf{I}}

 \numberwithin{equation}{section}
 \newtheorem{theorem}[equation]{Theorem}
 
 \newtheorem{lemma}[equation]{Lemma}
 \newtheorem{prop}[equation]{Proposition}
 \theoremstyle{definition}
 \newtheorem{definition}[equation]{Definition}
 \theoremstyle{remark}
 \newtheorem{rem}[equation]{Remark}

\journal{Journal of Combinatorial Theory A}

\begin{document}

\begin{frontmatter}



\title{New Bounds for Partial Spreads of $\h(2d-1, q^2)$ and Partial Ovoids of the Ree-Tits Octagon}

\author{Ferdinand Ihringer}
\ead{Ferdinand.Ihringer@gmail.com}
\address{Einstein Institute of Mathematics, The Hebrew University of Jerusalem, Givat Ram, Jerusalem, 9190401, Israel}

\author{Peter Sin}
\ead{sin@ufl.edu}
\address{Peter Sin, Department of Mathematics, University of Florida, Gainesville, FL 32611, USA}

\author{Qing Xiang}
\ead{qxiang@udel.edu}
\address{Qing Xiang, Department of Mathematical Sciences, University of Delaware, Newark, DE 19716, USA}

\begin{abstract}
Two results are obtained that give upper bounds on partial spreads and partial ovoids respectively.

The first result is that the size of a partial spread of the
Hermitian polar space $\h(3, q^2)$ is at most $\left(\frac{2p^3+p}{3} \right)^t+1$, where $q=p^t$, $p$ is a prime. For fixed $p$ this bound is in $o(q^3)$, which is asymptotically better than the previous best known bound of $(q^3+q+2)/2$. Similar bounds for partial spreads of $\h(2d-1, q^2)$, $d$ even, are given.
 
The  second result is that the size of a partial ovoid of the Ree-Tits octagon $\sfO(2^t)$ is at most  $26^t+1$. This bound, in particular, shows that the Ree-Tits octagon $\sfO(2^t)$ does not have an ovoid.
\end{abstract}


\end{frontmatter}

\section{Introduction}

Determining clique numbers of graphs is a traditional topic of combinatorics. Partial spreads of polar spaces are cliques of the oppositeness graph defined on the generators of the polar space. While bounds for partial spreads and partial ovoids are a long studied topic since Thas \cite{Thas1981} popularized the problem, many questions remain open. 
The purpose of this paper is to address two of these open questions.

Our first result gives an improved bound on the size of partial spreads of 
certain Hermitian polar spaces.
An \emph{Hermitian polar space}  $\h(2d-1, q^2)$, for $q=p^t$ a prime power, is the incidence geometry arising from a non-degenerate Hermitian form $f$ of $\bbF_{q^2}^{2d}$. Here the flats of $\h(2d-1, q^2)$ consist of all nonzero totally isotropic subspaces of $\bbF_{q^2}^{2d}$ with respect to the form $f$; incidence is the inclusion relation of the flats of $\h(2d-1, q^2)$. The maximal totally isotropic subspaces of $\bbF_{q^2}^{2d}$ with respect to the form $f$ are called \emph{generators} of $\h(2d-1, q^2)$. A  \emph{partial spread} of $\h(2d-1, q^2)$ is a set of pairwise disjoint generators of $\h(2d-1, q^2)$. A simple double counting argument shows that a partial spread of $\h(2d-1, q^2)$  has size at most $q^{2d-1}+1$. 
When $d$ is odd, a better upper bound of $q^d+1$ is known \cite{Vanhove2009}, and partial spreads of that size exist \cite{Aguglia2003}, \cite{Luyckx2008}.
So we are interested in the case when $d$ is even.
 
\begin{theorem}\label{thm_new_hermitian}
  Let $q = p^t$ with $p$ prime and $t\geq 1$.
  Let $Y$ be a partial spread of $\h(2d-1, q^2)$, where $d$ is even.
  \begin{enumerate}[(a)]
   \item If $d = 2$, then $|Y| \leq \left(\frac{2p^3+p}{3} \right)^t+1$.
   \item If $d = 2$ and $p=3$, then $|Y| \leq 19^t$.
   \item If $d > 2$, then $|Y|\leq \left(p^{2d-1} - p\frac{p^{2d-2}-1}{p+1}\right)^t+1$.
  \end{enumerate}
\end{theorem}

The previous best known bound is in $\Theta(q^{2d-1})$ \cite{DeBeule2008,Ihringer2014}, while even for $d=2$ the largest known examples only have size $\Theta(q^2)$ \cite{Coolsaet2014}. Here we provide the first upper bound which is in $o(q^{2d-1})$ for fixed $p$. 
For $d=2$ the previous best known bound is $(q^3+q+2)/2$ \cite{DeBeule2008}.
An easy calculation shows that this old bound is better than the bound in part (a) of Theorem \ref{thm_new_hermitian} if $p=2$ and $t \leq 2$ or if $t=1$. But for fixed $p$ (and let $q=p^t$), the bound in part (a) of Theorem \ref{thm_new_hermitian} is in $o(q^3)$, which is asymptotically better than the bound of $(q^3+q+2)/2$. For $d>2$ the new bound improves all previous bounds if $t > 1$.

Our second result is a bound on the size of a partial ovoid in the \emph{Ree-Tits octagons}. 
A {\it generalized $n$-gon} of order $(s, r)$ is a triple $\Gamma = (\scrP, \scrL, \sfI)$, where elements of $\scrP$ are called \emph{points}, 
elements of $\scrL$ are called \emph{lines}, and $\sfI \subseteq \scrP \times \scrL$ is an \emph{incidence relation} between the points and lines,
which satisfies the following axioms \cite{vanMaldeghem1998}:
\begin{enumerate}[(a)]
 \item Each line is incident with $s+1$ points.
 \item Each point is incident with $r+1$ lines.
 \item The \emph{incidence graph} has diameter $n$ and girth $2n$.
\end{enumerate}
Here the incidence graph is the bipartite graph with $\scrP \cup \scrL$ as vertices, $p\in \scrP$ and $\ell\in\scrL$ are adjacent if $(p,\ell)\in \sfI$.

The only known thick finite generalized octagons are the \emph{Ree-Tits octagons} $\sfO(2^t)$, $t$ odd, defined in \cite{Tits1983}, and their duals.

A \emph{partial ovoid} of a generalized $n$-gon $\Gamma$ is a set of points pairwise at distance $n$ in the incidence graph. 
An easy counting argument shows that the size of  a partial ovoid of a generalized octagon of order $(s, r)$ is at most $(sr)^2+1$. A partial ovoid of a generalized octagon of order $(s, r)$ is called an {\it ovoid} if it has the maximum possible size $(sr)^2+1$. The Ree-Tits octagon $\sfO(2^t)$ is a generalized octagon of order $(2^t, 4^t)$, so the size of an ovoid is $64^t+1$.
We shall establish the following bound.

\begin{theorem}\label{thm_new_reetits}
  The size of a partial ovoid of the Ree-Tits octagon $\sfO(2^t)$, $t$ odd, 
is at most $26^t+1$.
\end{theorem}

Theorem~\ref{thm_new_reetits} shows that the Ree-Tits octagon $\sfO(2^t)$ does not have an ovoid.  This had been conjectured by Coolsaet and Van Maldeghem \cite[p.~108]{Coolsaet2000}. For $\sfO(2)$, it was already shown \cite{Coolsaet2000} that a partial ovoid has at most 27 points, while for $t\geq 3$ no bounds better than the ovoid size were known. 
A computer search with the mixed-integer linear programming solver Gurobi showed that for $\sfO(2)$ the largest partial ovoid has size $24$, so the bound of Theorem \ref{thm_new_reetits} is not tight.
Note an earlier version of this document claimed that the Ree-Tits octagon, based on \cite[p. 582]{Parkinson2015}, does not possess a (non-trivial) line-domestic collineation.
This is corrected in \cite[Remark 2.3]{Parkinson2020}.

Both theorems are proved using the same approach based on the following simple and well known observation.

\begin{lemma}\label{lem:prank_bound}
  Let $(X, \sim)$ be a graph. Let $A$ be the adjacency matrix of $X$. 
  Let $Y$ be a clique of $X$. Then
  \begin{align*}
    |Y| \leq 
    \begin{cases}
     \rank_p(A) + 1, & \text{ if } p \text{ divides } |Y| - 1,\\
     \rank_p(A), & \text{ otherwise.}
    \end{cases}
  \end{align*}
\end{lemma}
\begin{proof}
  Let $J$ be the all-ones matrix of size $|Y| \times |Y|$.
  Let $I$ be the identity matrix of size $|Y| \times |Y|$.
  As $Y$ is a clique, the submatrix $A'$ of $A$ indexed by $Y$ is $J-I$.
  Hence, the submatrix has $p$-rank $|Y|-1$ if $p$ divides $|Y|-1$, and it has $p$-rank $|Y|$ if $p$ does not divide $|Y|-1$. As $\rank_p(A') \leq \rank_p(A)$,
  the assertion follows.
\end{proof}

The feasibility of applying Lemma~\ref{lem:prank_bound}  depends, 
of course, on the ability to compute or to bound the $p$-ranks of the adjacency matrix
in specific cases. In our cases, representation theory in characteristic
$p$ furnishes the necessary computations. 
The oppositeness graphs that we shall consider are  defined by maximal distance
in  association schemes arising from the Hermitian dual polar spaces and 
generalized octagons.  These maximal distance relations are  specific examples
of the general notion of oppositeness for flags in Tits buildings. It was
shown in \cite{Sin2012} that if $G$ is a finite group of Lie type 
in characteristic $p$, then an oppositeness relation in the associated building
defines a $G$-module homomorphism (in characteristic $p$) whose image is a 
certain simple module for $G$ and for a related semisimple algebraic group.
Crucially, the calculation of the dimension of this module
can be reduced to the prime field case (see Proposition~\ref{prop:prank_nonprime}).
In the prime case computations have been made for many simple modules, 
including those needed for this paper (\cite{Arslan2011}, \cite{Veldkamp1970}),  
as part of the representation theory of semisimple algebraic groups. In this way, we obtain the $p$-ranks of the oppositeness graphs 
for $\sfO(2^t)$ and $\h(3, q^2)$.  In the other cases of
$\h(2d-1, q^2)$ for $d>2$, a bound on the $p$-rank is obtained by relating the oppositeness matrix with an idempotent matrix of the association scheme. 

\section{Two Association Schemes}
A complete introduction to association schemes can be found in \cite[Ch. 2]{Brouwer1989}.

\begin{definition}
  Let $X$ be a finite set of size $n$. An \textit{association scheme} with $d+1$ classes is a pair $(X, \scrR)$, where $\scrR = \{ R_0, \ldots, R_d \}$ is a set of symmetric binary relations on $X$ with the following properties:
  \begin{enumerate}[(a)]
    \item $\scrR$ is a partition of $X \times X$.
    \item $R_{0}$ is the identity relation.
    \item There are numbers $p_{ij}^k$ such that for $x, y \in X$ with $x R_k y$ there are exactly $p_{ij}^k$ elements $z$ with $x R_i z$ and $z R_j y$.
  \end{enumerate}
\end{definition}

The relations $R_i$ are described by their \textit{adjacency matrices} $A_i \in \bbC^{n,n}$ defined by
\begin{align*}
  (A_i)_{xy} = \begin{cases}
                 1, & \text{ if } x R_i y,\\
                 0, & \text{ otherwise.}
               \end{cases}
\end{align*}
In this paper the matrix $A_d$ is referred to as the \emph{oppositeness matrix}.
Denote the all-ones matrix by $J$.
There exist (see \cite[p. 45]{Brouwer1989}) idempotent Hermitian matrices $E_j \in \bbC^{n,n}$ with the properties
\begin{align*}
\begin{array}{lll}\displaystyle
\sum_{j=0}^d E_j = I, & \hspace*{2cm} &\displaystyle E_{0} = n^{-1} J,\\
\displaystyle A_j = \sum_{i=0}^d P_{ij} E_i, &  &\displaystyle E_j = \frac{1}{n} \sum_{i=0}^d Q_{ij} A_i,
  \end{array}
\end{align*}
where $P = (P_{ij}) \in \bbC^{d+1,d+1}$ and $Q = (Q_{ij}) \in \bbC^{d+1,d+1}$ are the so-called eigenmatrices of the association scheme. 
The $P_{ij}$ are the eigenvalues of $A_j$.
The multiplicity $f_i$ of $P_{ij}$ satisfies
\begin{align*}
  f_i = \rank(E_i) = \tr(E_i) = Q_{0i}.
\end{align*}

In this paper we consider the association schemes corresponding to the dual polar graph of $\h(2d-1, q^2)$ and the Ree-Tits octagon.
For the dual polar graph of $\h(2d-1, q^2)$ we have the following situation. Here $X$ is the set of generators of $\h(2d-1, q^2)$ and two generators $a, b$ are in relation $R_i$, where $0\leq i\leq d$, if and only if $a$ and $b$ intersect in codimension $i$. The dual polar graph of $\h(2d-1, q^2)$ has diameter $d$ and is distance regular.

As in \cite{Ihringer2014}, for the dual polar graph of $\h(2d-1, q^2)$ we obtain
\begin{align*}
 f_d = q^{2d} \frac{q^{1-2d}+1}{q+1} = q^{2d-1} - q \frac{q^{2d-2}-1}{q+1},
\end{align*}
and
\begin{align*}
  Q_{id} = \frac{P_{di}}{P_{0i}} Q_{0d} = f_d (-q)^{-i}.
\end{align*}

For the Ree-Tits octagon $\sfO(2^t)$ we consider the following association scheme: the set $X$ is the set of points of the octagon,
and two points $a$ and $b$ are in relation $R_i$, where $0\leq i\leq 4$, if their distance in the incidence graph of $\sfO(2^t)$ is $2i$. In this case $A_4$ is the oppositeness
matrix. 

\section{$p$-Ranks of the Oppositeness Matrices}
The following result reduces the computation of $p$-ranks of oppositeness matrices
to the case of the prime field.  It is an application of 
Steinberg's Tensor Product Theorem.  
 
\begin{prop}[{Sin \cite[Prop. 5.2]{Sin2012}}]\label{prop:prank_nonprime}
Let $G(q)$, $q=p^t$ a prime power, be a finite group of Lie type and let $A(q)$ denote the oppositeness matrix for objects of a fixed self-opposite type 
in the building of $G(q)$. 
Then
  \begin{align*}
    \rank_p(A(q)) = \rank_p(A(p))^t.
  \end{align*}
\end{prop}
\qed
\begin{rem} The meaning of $q$ in Proposition~\ref{prop:prank_nonprime} needs to be explained. For (untwisted) Chevalley groups $q$ is simply the cardinality of the field used
to define the group.
In the case of Ree groups of type $\sfF_4$, there is a Steinberg endomorphism
$\tau$ of the algebraic group $\sfF_4$ (over an algebraic closure of $\bbF_2$)
such that the Ree group that we denote by $G(q)$ is the subgroup of fixed points
of $\tau$, and the subgroup of fixed points of $\tau^2$ is the Chevalley
group $\sfF_4(q)$, where $q$ is an odd power of $2$.  
In general, the relation between $G(q)$ and $q$ 
is that the highest weight of the simple module in \cite[Theorem 4.1]{Sin2012} 
is of the form $(q-1)\tilde\omega$, where $\tilde\omega$ is a sum of fundamental weights.
We refer to \S5 of \cite{Sin2012} for a detailed description. 
\end{rem}

\subsection{Proof of Theorem~\ref{thm_new_hermitian} }


\begin{lemma}\label{lem:prank_prime}
  \begin{enumerate}[(a)]
   \item The $p$-rank of the oppositeness matrix of lines of $\h(3, p^2)$ is $\frac{2p^3+p}{3}$.
   \item The $p$-rank of the oppositeness matrix of generators of $\h(2d-1, p^2)$, $d$ even, is at most $p^{2d-1} - p\frac{p^{2d-2}-1}{p+1}$.
  \end{enumerate}
\end{lemma}
\begin{proof}
  Notice that $\h(3, p^2)$ is dual to $\q^-(5, p)$, so the oppositeness matrix
of lines  of of $\h(3, p^2)$ is the oppositeness matrix of points in $\q^-(5, p)$. 
  By \cite[Example 6.2]{Sin2012}, the $p$-rank of this oppositeness matrix is equal to the dimension of the simple module for a simply connected algebraic group
of type $D_3$ with highest weight equal to $(p-1)\omega_1$, under the standard parametrization of simple modules, where $\omega_1$ is the first fundamental weight. This dimension was calculated in \cite[Theorem 1.2]{Arslan2011}.
  Applying Theorem 1.2 of \cite{Arslan2011} with $\ell = 3$ and $r=p-1$, for the oppositeness matrix $A_2$ we obtain
  \begin{align*}
    \rank_p(A_2) &= \frac{2p^3+p}{3}.
  \end{align*}
  Part (a) of the lemma follows.
  
  For (b) notice that the matrix $E_d$ of the dual polar graph of $\h(2d-1, p^2)$ has rank
  \begin{align*}
    f_d = p^{2d-1} - p\frac{p^{2d-2}-1}{p+1}.
  \end{align*}
When $d$ even the matrix $n p^{d-1} E_d$ has only integer entries and we have $A_d \equiv n p^{d-1} E_d \mod p$. Hence,  
$\rank_p(A_d)=\rank_p(n p^{d-1} E_d)\leq\rank(n p^{d-1} E_d)=\rank(E_d) = p^{2d-1} - p\frac{p^{2d-2}-1}{p+1}$.
\end{proof}

Theorem \ref{thm_new_hermitian} now follows by combining Lemma~\ref{lem:prank_prime}, 
Proposition \ref{prop:prank_nonprime} and Lemma \ref{lem:prank_bound}. 
Notice here that $\frac{2p^3+p}{3}$ is divisible by $p$ unless $p=3$.
\qed

\subsection{Proof of  Theorem~\ref{thm_new_reetits}}
Let $\{\alpha_1, \alpha_2, \alpha_3, \alpha_4\}$ be a set of fundamental
roots for a root system of type $\sfF_4$, where $\alpha_1$ is a long root
at one end of the Dynkin diagram and $\alpha_4$ is the short root at the other end.
According to \cite[p. 282, para. 3]{Sin2012}, the $2$-rank of the oppositeness
matrix of $\sfO(2)$ is equal to the dimension of the simple module of the group $\sfF_4$ whose highest weight is the fundamental weight corresponding to $\alpha_4$. The dimension can be found in several ways. First, Veldkamp calculated all simple modules of the group $\sfF_4$ in characteristic $2$. (See \cite[Table 2]{Veldkamp1970}; the highest weight in question is denoted $d_4=0001$). Secondly, this module happens to be the modulo 2 reduction of the corresponding Weyl module in characteristic zero, so its dimension is also given by Weyl's Dimension Formula. Finally, it is also easy to compute the oppositeness matrix of $\sfO(2)$ and its 2-rank by computer. 
Hence, we obtain the following result.
\begin{lemma}\label{lem:prank_reetits}
  The $2$-rank of the oppositeness matrix of $\sfO(2)$ is $26$.
\end{lemma}

Theorem~\ref{thm_new_reetits} now follows by combining Proposition~\ref{prop:prank_nonprime}, Lemma~\ref{lem:prank_reetits} and Lemma~\ref{lem:prank_bound}.
\qed

\section{Final Remarks}
The association scheme technique used also work for partial ovoids of the twisted triality hexagon of order $(q, q^3)$, where we obtain the $p$-rank $\left((4p^5+p)/5\right)^t$,
but there a better bound of order $q^3+1$ is known \cite[Theorem 6.4.19]{Vanhove2011}.
For all other known generalized polygons $p$-ranks do not improve any known bounds for partial ovoids.

The bounds on partial spreads $\h(2d-1, q^2)$, $d > 2$ even, in Theorem \ref{thm_new_hermitian} could be improved by calculating exact $p$-ranks. We would expect the bounds
thus obtained to be much better than the ones given in Theorem~\ref{thm_new_hermitian}, 
were we to extrapolate from the case of  $\h(5, q^2)$. There, the corresponding $p$-rank is $\left((11p^5+5p^3+4p)/20\right)^t$, while an argument as in Lemma \ref{lem:prank_prime} (b) only yields $\left(p^5-p^4+p^3-p^2+p\right)^t$ as an upper bound. 
Unfortunately, calculating the exact $p$-rank seems to be too complicated. Notice however that this is a finite problem for fixed $p$ and fixed $d$ as one only has to calculate the $p$-rank for $t=1$ by Proposition \ref{prop:prank_nonprime}.

\paragraph*{\bf Acknowledgments} The first author thanks John Bamberg and Klaus Metsch for their remarks.
The second author was partially supported by a grant from the Simons Foundation (\#204181 to Peter Sin). The third author was partially supported by an NSF grant DMS-1600850.


\end{document}